\documentclass[letter,11pt]{amsart}

\usepackage{amsmath,amsthm,amssymb,mathtools,thmtools,mathabx}
\usepackage{cite}
\usepackage{upgreek}
\usepackage[bookmarks=false]{hyperref}
\usepackage{enumitem}
\usepackage[margin=0.9in]{geometry}

\theoremstyle{plain}
\newtheorem{theorem}{Theorem}
\newtheorem{lemma}[theorem]{Lemma}

\newtheorem{corollary}[theorem]{Corollary}
\theoremstyle{definition}
\newtheorem{definition}[theorem]{Definition}
\newtheorem{remark}[theorem]{Remark}

\newcommand{\cC}{\mathcal{C}}

\newcommand{\bF}{\mathbb{F}}
\newcommand{\cG}{\mathcal{G}}

\newcommand{\cJ}{\mathcal{J}}

\newcommand{\N}{\mathbb{N}}

\newcommand{\cU}{\mathcal{U}}

\newcommand{\Z}{\mathbb{Z}}

\newcommand{\eps}{\varepsilon}
\newcommand{\vphi}{\varphi}

\newcommand{\norm}[1]{\left\|#1\right\|}
\newcommand{\abs}[1]{\left|#1\right|}

\newcommand{\ot}{\otimes}

\DeclareSymbolFont{Symbols}{OMS}{cmsy}{m}{n}
\DeclareMathSymbol{\Emptyset}{\mathord}{Symbols}{"3B}

\begin{document}

\title{Hilbert--Schmidt stability for graph products}

\begin{abstract}
    In this short note we prove Hilbert--Schmidt stability for graph products of abelian groups and $C^*$-algebras on chordal graphs. In particular, this shows that right-angled Artin groups on chordal graphs are Hilbert--Schmidt stable.
\end{abstract}

\author{Pieter Spaas}\thanks{The author was partially supported by the European Union via an ERC grant (QInteract, Grant No 101078107).}
\address{Centre for the Mathematics of Quantum Theory, Department of Mathematical Sciences, University of Copenhagen, Universitetsparken 5, DK-2100 Copenhagen \O, Denmark}
\email{pisp@math.ku.dk}

\
\vspace{-20pt}

\maketitle

\section*{Introduction}

Stability, which in general asks the question whether elements that ``almost'' satisfy a certain equation are ``close'' to elements that exactly satisfy it, has a long history and traces back to the natural question whether two approximately commuting matrices (in some norm) are close to a pair of commuting matrices; see, for instance, the introduction to \cite{Io21} and the references therein. 
In this note, we are mainly interested in \textit{HS-stability} for groups, where a group is called HS-stable if approximate homomorphisms into unitary groups with the normalized Hilbert--Schmidt norm are necessarily close to genuine homomorphisms (see Definition~\ref{def:HSstabGrp}). HS-stability for amenable groups, while still presenting many challenges, has seen considerable progress in recent years, see for instance \cite{HS17, ES22, LV22}. Further, it is known that HS-stability is preserved under free products as well as direct products with amenable HS-stable groups \cite{IS19}. However, beyond groups obtained from amenable HS-stable groups through these constructions, very few examples of nonamenable HS-stable groups are known. Some examples are certain one-relator groups with non-trivial center \cite{HS17}, certain specific graph products \cite{At18}, and virtually free groups \cite{GS21}. 

This note adds to these examples by showing that all right-angled Artin groups on chordal graphs are HS-stable, and more generally, that graph products of abelian groups on chordal graphs are HS-stable; see Corollary~\ref{cor:stabchordgrp}. In fact, one can allow some vertex groups to be HS-stable but not necessarily abelian; we refer to Theorem~\ref{thm:general} below for the most general statement. The latter theorem, as well as our main Theorem~\ref{thm:stabchord}, are phrased in the language of $C^*$-algebras, and they yield the aforementioned result on groups by passing to the full group $C^*$-algebra.

Finally, we also refer the reader to \cite{FW26}, where the authors obtain many new examples of groups with the (local) lifting property and the (R)FD property, and connect it to so-called very flexible HS-stability. Their results include very flexible HS-stability for right-angled Artin groups on chordal graphs; in particular, with a completely different proof, we strengthen this to HS-stability.

\section*{Preliminaries}

\subsection*{Stability} In the following, we denote by $U(n)$ the unitary group of dimension $n$, and by $\norm{\cdot}_2$ the normalized Hilbert--Schmidt norm.

\begin{definition}\label{def:HSstabGrp}
    Let $\Gamma$ be a countable discrete group. A sequence of maps $\vphi_n:\Gamma\to U(k_n)$ is called an \textit{approximate homomorphism} if $\lim_{n\to\infty}\norm{\vphi_n(gh) - \vphi_n(g)\vphi_n(h)}_2 = 0$ for all $g,h\in\Gamma$. We say $\Gamma$ is \textit{Hilbert--Schmidt stable} (or \textit{HS-stable}) if for every such asymptotic homomorphism, there exists a sequence of homomorphisms $\psi_n:\Gamma\to U(k_n)$ such that $\lim_{n\to\infty} \norm{\vphi_n(g) - \psi_n(g)}_2 = 0$ for all $g\in \Gamma$.
\end{definition}

For our main result Theorem~\ref{thm:stabchord} below, we will use the fact that stability can conveniently be described in the language of ultraproducts. In this note, when taking an ultraproduct of tracial $C^*$-algebras (e.g., matrix algebras with their normalized trace), we always mean the \textit{tracial ultraproduct}: for tracial $C^*$-algebras $(A_i,\tau_i)_{i\in I}$ and an ultrafilter $\cU$ on $I$, the tracial ultraproduct $\prod_\cU (A_i, \tau_i)$ is the quotient of the $C^*$-product $\prod_{i\in I} A_i$ by the ideal $\cJ = \{(a_i)_i\in \prod_{i\in I} A_i: \lim_{\cU} \tau_i(a_i^*a_i) = 0\}$. 

When the tracial states are implicit or clear from context, we will just write $\prod_\cU A_i$. We further define the norms $\norm{x}_{2,i}^2 = \tau_i(x^*x)$ for $x\in A_i$ and $\norm{x}_{2,\cU} = \lim_\cU \norm{x_i}_{2,i}$ for $x = (x_i)_i\in \prod_\cU A_i$. 

\begin{definition}
	Let $A$ be a $C^*$-algebra and $(A_i,\tau_i)_{i\in I}$ be a family of $C^*$-algebras equipped with tracial states. Let $\cU$ be an ultrafilter on $I$. We say a $^*$-homomorphism $\theta:A\to\prod_\cU A_i$ \textit{lifts} if there exist $^*$-homomorphisms $\theta_i:A\to A_i$ for every $i\in I$ such that $\theta = (\theta_i)_{i\in \cU}$, i.e., for every $a\in A$ we have $\theta(a) = (\theta_i(a))_{i\in \cU}$.
\end{definition}

\begin{definition}\label{def:tracialstab}
	Given a class $\cC$ of tracial $C^*$-algebras, we say a $C^*$-algebra $A$ is \textit{$\cC$-tracially stable} if every $^*$-homomorphism $\theta:A\to\prod_\cU A_i$, with every $A_i\in\cC$, lifts. In the case where $\cC$ consists of all matrix algebras equipped with their canonical normalized traces, we also refer to $\cC$-tracial stability as \textit{Hilbert--Schmidt stability} (or \textit{HS-stability}).
\end{definition}

Tracial stability for $C^*$-algebras as defined above can be interpreted in terms of $\eps$'s, $\delta$'s, and approximate $^*$-homomorphisms, see for instance \cite[Section~2]{HS16}. Moreover, when considering the full $C^*$-algebra $C^*(\Gamma)$ of a countable discrete group $\Gamma$, it is easy to check that HS-stability of $\Gamma$ as defined in Definition~\ref{def:HSstabGrp} is equivalent to HS-stability of $C^*(\Gamma)$ as defined in Definition~\ref{def:tracialstab}. Furthermore, by writing out concretely the definitions of tracial ultraproducts and lifting of $^*$-homomorphisms, one can easily prove the following criterion to check whether a $^*$-homomorphism from a $C^*$-algebra into an ultraproduct lifts.

\begin{lemma}[\!\!\!{\cite[Lemma~2.2]{HS16}}]\label{lem:HS}
	Suppose $A$ is generated as a $C^*$-algebra by elements $\{b_1, b_2, \ldots\}$. Let $(A_i,\tau_i)_{i\in I}$ be a family of tracial $C^*$-algebras, let $\cU$ be an ultrafilter on $I$, and let $\theta:A\to\prod_\cU (A_i,\tau_i)$ be a $^*$-homomorphism. Then the following are equivalent:
	\begin{enumerate}[label=(\roman*),nolistsep]
		\item $\theta$ lifts.
		\item For every $\eps>0$ and every $n\in\N$, there are $^*$-homomorphisms $\theta_i:A\to A_i$ for all $i\in I$ such that, for every $1\leq k\leq n$,
		\[
		\norm{(\theta_i(b_k))_\cU - \theta(b_k)}_{2,\cU} <\eps.
		\]
	\end{enumerate}
\end{lemma}

\subsection*{Graph products} Motivated by results of \cite{At18}, we study stability of \textit{graph products} of groups, first introduced and studied in general in \cite{Gr90}. In this note, all our graphs will be finite and simple, i.e., there are no loops or multiple edges. For completeness, we include the definition of a graph product group:

\begin{definition}
    Let $\cG = (V,E)$ be a graph, and for each vertex $v\in V$, let $\Gamma_v$ be a countable discrete group. The \textit{graph product group} $\Asterisk_\cG\, \Gamma_v$ is defined as the quotient of the free product $\ast_{v\in V} \Gamma_v$ by the normal subgroup generated by $\{ghg^{-1}h^{-1}: g\in \Gamma_v, h\in \Gamma_w, (v,w)\in E\}$. 
\end{definition}

In other words, the subgroups $\Gamma_v$ and $\Gamma_w$ of the graph product commute if and only if $v$ and $w$ are adjacent in $\cG$. Particular examples of interest, with applications in various fields across mathematics, are right-angled Artin groups and right-angled Coxeter groups, which arise as graph products where all vertex groups are $\Z$ and $\Z/2\Z$, respectively.

For $C^*$-algebras, the (full) graph product is defined similarly using a universal property:

\begin{definition}\label{def:Cgraphprod}
    Let $\cG = (V,E)$ be a graph, and for each vertex $v\in V$, let $A_v$ be a unital $C^*$-algebra. The \textit{graph product $C^*$-algebra} $\Asterisk_\cG A_v$ is defined as the unique unital $C^*$-algebra $\Asterisk_\cG A_v$ together with unital $^*$-homomorphisms $\iota_v: A_v \to \Asterisk_{\cG} A_v$ satisfying:
    \begin{itemize}[nolistsep]
        \item $[\iota_v(a), \iota_w(b)] = 0$ whenever $a\in A_v$, $b\in A_w$, $(v,w)\in E$;
        \item For any unital $C^*$-algebra $B$ together with $^*$-homomorphisms $\{\pi_v:A_v\to B\}_{v\in V}$ such that $[\pi_v(a),\pi_w(b)]=0$ whenever $a\in A_v$, $b\in A_w$, and $(v,w)\in E$, there exists a unique $^*$-homomorphism $\Asterisk_\cG\, \pi_v: \Asterisk_\cG A_v \to B$ such that $\Asterisk_\cG \, \pi_v \circ \iota_v = \pi_v$ for every $v\in V$.
    \end{itemize}
\end{definition}

In our main result Theorem~\ref{thm:stabchord} below, we will consider graph products on \textit{chordal graphs}: A finite simple graph is called chordal if every cycle of length at least $4$ has a chord, i.e., an edge which is not part of the cycle connecting two of its vertices. In other words, there are no induced subgraphs which are cycles of length at least $4$. 

Chordal graphs are a well-studied class of graphs. One characterization of chordal graphs that will be useful to us, is the following well-known fact.

\begin{lemma}[Inductive Construction of Chordal Graphs]\label{lem:chordal} 
    Chordal graphs $\cG$ (with $n$ vertices) are exactly the graphs that can be built in the following inductive way.
    \begin{enumerate}[leftmargin=*]
        \item Start with the trivial graph $\cG_1$ with one vertex and no edges.
        \item For $2\leq k\leq n$: Choose a complete (possibly empty) subgraph of $\cG_{k-1}$. To obtain $\cG_k$, add one vertex, and connect it with an edge to every vertex of the chosen complete subgraph of $\cG_{k-1}$.
        \item Put $\cG\coloneqq\cG_n$.
    \end{enumerate}
\end{lemma}

\section*{Results}

We now prove the following theorem about stability of graph product $C^*$-algebras, where we recall the well-known fact that every unital separable commutative $C^*$-algebra arises as the complex-valued functions on a compact Hausdorff space.  The proof essentially follows \cite[Theorem~2.9]{At18} (and \cite[Theorem~2.7]{HS16}, which in graph product language is the case of the graph with two vertices connected by an edge), though we are able to generalize their proofs to chordal graphs, as well as simplify it. We also note that we pay special attention to extending the $^*$-homomorphisms to the corresponding GNS von Neumann algebras when necessary to deal carefully with the projections that show up in the proof.

\begin{theorem}\label{thm:stabchord}
	Let $\cG=(V,E)$ be a chordal graph, and for every vertex $v\in V$, let $A_v=C(X_v)$ be a unital separable commutative $C^*$-algebra. Then the graph product $C^*$-algebra $\Asterisk_\cG A_v$ is HS-stable.
\end{theorem}
\begin{proof}
We proceed by induction on the number of vertices $\abs{V}$ in $\cG$. If $\abs{V}=1$, the result follows from \cite[Theorem~2.5]{HS16}. Suppose $\abs{V}\geq 2$, and construct $\cG$ inductively as in Lemma~\ref{lem:chordal}.

Assume $\theta:\Asterisk_\cG A_v\to \prod_\cU M_i$ is a unital $^*$-homomorphism, where each $M_i$ is a matrix algebra equipped with its canonical normalized trace. Note that $\Asterisk_\cG A_v$ is generated as a $C^*$-algebra by (countably many) contractions $b\in A_v$ for $v\in V$, where we view $A_v \simeq \iota_v(A_v)$ as a subalgebra of $\Asterisk_\cG A_v$. Fix $\eps>0$ and contractions $b_1, \ldots, b_n$ with $b_k\in A_{v_k}$ for some $v_k\in V$. By Lemma~\ref{lem:HS}, it suffices to find $^*$-homomorphisms $\theta_i:\Asterisk_\cG A_v\to M_i$ such that
\begin{equation}\label{eq:goal}
\norm{(\theta_i(b_k))_\cU - \theta(b_k)}_{2,\cU} < \eps,
\end{equation}
for every $1\leq k\leq n$. 

Let $v_0$ denote the vertex which was added during the last step of the inductive construction of $\cG$ from Lemma~\ref{lem:chordal}, and let $\cG_0=(V_0,E_0)$ denote the complete subgraph consisting of the neighbors of $v_0$. In particular, $\Asterisk_{\cG_0} A_v = \ot_{v\in V_0} A_v$ is a commutative sub-$C^*$-algebra of $\Asterisk_\cG A_v$, and thus the image $\theta(\Asterisk_{\cG_0} A_v)$ is commutative. 

For each $v\in V_0$, we list the elements $a_{v,1}, \ldots, a_{v,n_v}$ among the $b_k$'s that belong to $A_v$. Being in commutative $C^*$-algebras, we can approximate the elements $a_{v,1}, \ldots, a_{v,n_v}$ for $v\in V_0$ (viewed as functions on $X_v$) by simple functions: there exist Borel partitions $\{E_{v,1}, \ldots, E_{v,m_v}\}$ of $X_v$ and points $x_{v,d}\in E_{v,d}$ such that
\[
\norm{a_{v,l} - \sum_{d=1}^{m_v} a_{v,l}(x_{v,d})\chi_{E_{v,d}}} < \eps,
\]
for all $v\in V_0$ and $1\leq l\leq n_v$, where $\norm{\cdot}$ denotes the uniform norm.

We observe that all our $C^*$-algebras come equipped with a canonical tracial state, namely the usual normalized traces $\tau_i$ on the matrix algebras $M_i$, $\tau_\cU \coloneqq \lim_\cU \tau_i$ on $\prod_\cU M_i$, and the pullback traces $\tau_\cG\coloneqq \tau_\cU\circ\theta$ on $\Asterisk_\cG A_v$ and $\tau_v\coloneqq \tau_\cU\circ\theta|_{A_v}$ on $A_v$ for each $v\in V$. 

By construction, the $^*$-homomorphism $\theta$ is trace-preserving, and thus extends to the von Neumann algebra completions $\pi_{\tau_v}(A_v)''$ in the GNS-construction for the corresponding tracial states. We also note that the tracial ultraproduct $\prod_\cU (M_i,\tau_i)$ is a von Neumann algebra itself, and thus we moreover have that
\[
\theta(\pi_{\tau_v}(A_v)'')\subset \prod_\cU (M_i,\tau_i).
\]
Since we can view $\chi_{E_{v,d}}\in \pi_{\tau_v}(A_v)''$, we can thus consider the corresponding projections 
\[
\theta(\chi_{E_{v,d}})\eqqcolon Q_{v,d}\in \prod_\cU (M_i,\tau_i),
\]
which, when ranging over $d$, form a partition of unity for every $v\in V_0$. Moreover, since they are all mutually commuting, we can consider the partition of unity $Q_1, \ldots, Q_r$, generated by all projections $Q_{v,d}$, for $v\in V_0$, $1\leq d\leq m_v$.

Next, denote by $\cG'=(V',E')$ the induced subgraph of $\cG$ obtained by removing $v_0$. Consider the graph product $\Asterisk_{\cG'} B_v$, where for each $v\in V'=V\setminus\{v_0\}$ we define
\begin{align*}
	B_v = \begin{cases}
		A_v &\text{if}\;\; v\notin V_0,\\
		C^*(\chi_{E_{v,1}},\ldots, \chi_{E_{v,m_v}}) &\text{if}\;\; v\in V_0.
	\end{cases}
\end{align*}
Define $\rho:\Asterisk_{\cG'} B_v\to \prod_\cU M_i$ by $\rho|_{B_v} = \theta|_{A_v}$ when $v\notin V_0$, and $\rho|_{B_v} = \theta|_{C^*(\chi_{E_{v,1}},\ldots, \chi_{E_{v,m_v}})}$ if $v\in V_0$, where we again assume that we extended $\theta$ to $\pi_{\tau_v}(A_v)''\supset C^*(\chi_{E_{v,1}},\ldots, \chi_{E_{v,m_v}})$.

Note that $\rho$ is well-defined, since by construction $\{\rho|_{B_v}\}_{v\in V'}$ satisfies the required commutation relations imposed by the graph (essentially since $\theta$ does), and thus the universal property from Definition~\ref{def:Cgraphprod} yields the desired $\rho$. Next, we observe that $\cG'$ is chordal, and thus by the induction hypothesis, $\Asterisk_{\cG'} B_v$ is HS-stable. In particular, we can lift $\rho$ to unital $^*$-homomorphisms $\rho_i:\Asterisk_{\cG'} B_v\to M_i$. 

Let $\rho_{v,i}\coloneqq \rho_i|_{B_v}$, and for each $v\in V_0$, $1\leq d\leq m_v$, and $i\in I$, let $P_{v,d}^i\coloneqq \rho_i(\chi_{E_{v,d}})$. Then all the $P_{v,d}^i$ mutually commute, and $\{P_{v,d}^i\}_{d=1}^{m_v}$ is a partition of unity in $M_i$ for every $v\in V_0$ and $i\in I$. Moreover, by construction, $Q_{v,d} = (P_{v,d}^i)_{i\in\cU}$ for every $v\in V_0$, $1\leq d\leq m_v$. Denote by $P_{i,1},\ldots, P_{i,r_i}$ the partition of unity of $M_i$ generated by all $P_{v,d}^i$. We note that we can w.l.o.g.\ assume that $r_i=r$ for all $i$. Then, for every $i\in I$, we can order the $P_{i,j}$ such that for all $1\leq j\leq r$, we have $Q_j=(P_{i,j})_\cU$.

Next, since $\Asterisk_{\cG_0} A_v$ is commutative and commutes with $A_{v_0}$, we observe that
\[
C^*(Q_1, \ldots, Q_r) = C^*\big(\cup_{v\in V_0} C^*(Q_{v,1},\ldots, Q_{v,m_v})\big) \subset \theta(C^*(A_{v_0},\Asterisk_{\cG_0} A_v))'\cap \prod_\cU M_i.
\]
Hence, by taking relative commutants,
\begin{align*}
    \theta(C^*(A_{v_0},\Asterisk_{\cG_0} A_v)) &\subset C^*(Q_1, \ldots, Q_r)'\cap \prod_\cU M_i \\
    &= \prod_\cU \big(\bigoplus_{j=1}^r P_{i,j}M_iP_{i,j}\big) \\
    &= \bigoplus_{j=1}^r \big(\prod_\cU P_{i,j}M_iP_{i,j}\big).
\end{align*}
Denote by $\vphi_j$ the orthogonal projection onto the summand $\prod_\cU P_{i,j}M_iP_{i,j}$ of $\bigoplus_{j=1}^r \prod_\cU P_{i,j}M_iP_{i,j}$. Since $A_{v_0}$ is commutative and hence HS-stable, we can lift $\vphi_j\circ\theta: A_{v_0}\to \prod_\cU P_{i,j}M_iP_{i,j}$ for each $1\leq j\leq r$ to $^*$-homomorphisms
\[
\theta_{v_0,i,j}: A_{v_0}\to P_{i,j}M_iP_{i,j}.
\]
We can thus define $\theta_{v_0,i}\coloneqq \bigoplus_{j=1}^r \theta_{v_0,i,j}$, and note that for every $a\in A_{v_0}$ we have
\begin{equation}\label{eq:0}
	(\theta_{v_0,i}(a))_\cU = \theta(a).
\end{equation}
For $v\in V_0$ and $a\in A_v$ we next define
\[
\theta_{v,i}(a) = \sum_{j=1}^r a(x_{v,d_j})P_{i,j},
\]
where $d_j$ is such that $P_{i,j}\leq P_{v,d_j}^i$, which is well-defined since $(P_{i,j})_j$ is by construction a finer partition than $(P_{v,d}^i)_d$. We note that by construction, we have for every $v\in V_0$ and $1\leq l\leq n_v$
\begin{align}
	\begin{split}\label{eq:1}
	\norm{(\theta_{v,i}(a_{v,l}))_\cU - \theta(a_{v,l})}_{2,\cU} &= \norm{\big(\sum_{j=1}^r a_{v,l}(x_{v,d_j})P_{i,j}\big)_\cU - \theta(a_{v,l})}_{2,\cU}\\
	&= \norm{\big(\sum_{d=1}^{m_v} a_{v,l}(x_{v,d})P_{v,d}^i\big)_\cU - \theta(a_{v,l})}_{2,\cU}\\
	&= \norm{\big(\sum_{d=1}^{m_v} a_{v,l}(x_{v,d})\rho_i(\chi_{E_{v,d}})\big)_\cU - \theta(a_{v,l})}_{2,\cU}\\
	&= \norm{\theta\big(\sum_{d=1}^{m_v} a_{v,l}(x_{v,d})\chi_{E_{v,d}}\big) - \theta(a_{v,l})}_{2,\cU}\\
	&\leq \norm{\sum_{d=1}^{m_v} a_{v,l}(x_{v,d})\chi_{E_{v,d}} - a_{v,l}}_{2,\cU}\\
	&< \eps.
	\end{split}
\end{align}
Finally, for $v\notin V_0\cup\{v_0\}$, we let $\theta_{v,i}\coloneqq \rho_{v,i}$, and note that since $(\rho_i)_{i\in I}$ was a lift of $\rho$, which equals $\theta$ on $A_v$ for $v\notin V_0\cup\{v_0\}$, we get for every $v\notin V_0\cup\{v_0\}$ and $a\in A_v$ that
\begin{equation}\label{eq:2}
(\theta_{v,i}(a))_\cU = \theta(a).
\end{equation}
Since by construction, $\{\theta_{v,i}\}_{v\in V}$ satisfies the necessary commutation relations for every $i\in I$, we can define the $^*$-homomorphisms
\[
\theta_i\coloneqq \Asterisk_\cG\, \theta_{v,i}:\Asterisk_\cG A_v\to M_i,
\]
for $i\in I$. From \eqref{eq:0}, \eqref{eq:1}, and \eqref{eq:2} we can now immediately conclude that
\begin{align*}
	\norm{(\theta_i(b_k))_\cU - \theta(b_k)}_{2,\cU} = \norm{(\theta_{v_k,i}(b_k))_\cU - \theta(b_k)}_{2,\cU} <\eps.
\end{align*}
for every $1\leq k\leq n$. Hence the $\theta_i$ satisfy \eqref{eq:goal}, which finishes the proof.
\end{proof}

\textbf{Remarks.}
\begin{enumerate}[label=(\roman*), align=left, leftmargin=*, labelsep=-5pt]
	\item Instead of matrix algebras, we could take the $M_i$ in the proof to belong to any class $\cC$ of unital real rank 0 $C^*$-algebras closed under direct sums and unital corners, and equipped with tracial states (cf.\ \cite{At18,HS16}). Hence for such classes $\cC$, Theorem~\ref{thm:stabchord} also implies $\cC$-tracial stability of graph products of commutative $C^*$-algebras on chordal graphs.
	\item In the proof, we only used the assumption of commutativity for the $C^*$-algebras at the neighboring vertices of the added vertex, while the $C^*$-algebra at this added vertex just needed to be HS-stable. Hence, the proof in fact implies the following more general result:
\end{enumerate}

\begin{theorem}\label{thm:general}
    A graph product of $C^*$-algebras on a chordal graph $\cG$ is HS-stable (respectively $\cC$-tracially stable) if every vertex $C^*$-algebra belonging to a vertex which appears in a complete subgraph from step~(2) of the inductive construction from Lemma~\ref{lem:chordal} of $\cG$ is commutative, and every other vertex $C^*$-algebra is HS-stable (respectively $\cC$-tracially stable).
\end{theorem}

By passing to the full group $C^*$-algebra, Theorem~\ref{thm:stabchord} immediately implies the following for graph products of countable groups. Note that the foregoing remarks also apply in the group setting.

\begin{corollary}\label{cor:stabchordgrp}
	Let $\cG=(V,E)$ be a chordal graph, and for every vertex $v\in V$, let $\Gamma_v$ be a countable abelian group. Then the graph product group $\Asterisk_\cG\, \Gamma_v$ is HS-stable. In particular, right-angled Artin groups and right-angled Coxeter groups on chordal graphs are HS-stable. 
\end{corollary}

\begin{remark}
    It is known that a right-angled Coxeter group is virtually free if and only if its graph is chordal (the author would like to thank Francesco Fournier-Facio for pointing this out). Therefore, the right-angled Coxeter case of Corollary~\ref{cor:stabchordgrp} also follows from \cite{GS21}, albeit with a completely different proof. 
\end{remark}

\begin{remark}
    We note that not all right-angled Artin groups are HS-stable. Indeed, Ioana showed in \cite{Io21} that the product of free groups $\bF_2\times\bF_2$, which is the right-angled Artin group on the square, is not HS-stable. In particular, since it is easy to check that retracts of stable groups are stable, right-angled Artin groups on graphs which have an induced subgraph equal to the square are not HS-stable. Whether right-angled Artin groups on cycles of length 5 or more are HS-stable seems to be open (and a resolution of this question would likely lead to a classification of HS-stable right-angled Artin groups).\\
    In the Coxeter case, it is known that the right-angled Coxeter group on the square (which is $D_\infty\times D_\infty$, where $D_\infty$ is the infinite dihedral group) is HS-stable. This follows for instance from the facts that $D_\infty\cong \Z/2\Z*\Z/2\Z$, finite groups are HS-stable, and that HS-stability is closed under free products as well as under the direct product with an amenable HS-stable group \cite[Corollary~D]{IS19}. On the other hand, whether for instance the right-angled Coxeter groups on the pentagon and on $K_{3,3}$ are HS-stable seems to be open. Note that the right-angled Coxeter group on $K_{3,3}$ is a direct product of two virtually free groups.
\end{remark}

\subsection*{Acknowledgments} The results in this note were obtained after some fruitful discussions of the author with Adrian Ioana and Scott Atkinson on related topics, for which the author would like to thank them. Additionally, the author would like to thank Francesco Fournier-Facio and Rufus Willett for sharing their preprint \cite{FW26} and for some comments on an earlier draft of this note.

\end{document}